\documentclass{article}

\usepackage{setspace}
\usepackage{fullpage}

\usepackage{makeidx}
\makeindex

\usepackage{amssymb}
\usepackage{amsmath}
\usepackage[all]{xy}
\usepackage{amsthm}

\theoremstyle{plain}

\numberwithin{equation}{section}

\theoremstyle{definition}

\newtheorem{definition}{Definition}[section]

\newtheorem{example}[definition]{Example}

\theoremstyle{remark}

\theoremstyle{plain}

\newtheorem{theorem}[definition]{Theorem}
\newtheorem{lemma}[definition]{Lemma}
\newtheorem{corollary}[definition]{Corollary}

\begin{document}

 \def\AW[#1]^#2_#3{\ar@{-^>}@<.5ex>[#1]^{#2} \ar@{_<-}@<-.5ex>[#1]_{#3}}
 \def\NAW[#1]{\ar@{-^>}@<.5ex>[#1] \ar@{_<-}@<-.5ex>[#1]}

\xymatrixrowsep{1.5pc} \xymatrixcolsep{1.5pc}

 \renewcommand{\O}{\bigcirc}
 \newcommand{\OX}{\bigotimes}
 \newcommand{\OD}{\bigodot}
 \newcommand{\OV}{\O\llap{v\hspace{.6ex}}}
 \newcommand{\B}{\mbox{\Huge $\bullet$}}
  \newcommand{\D}{$\diamond$}

\title{Classification of finite-growth contragredient Lie superalgebras}
\author{Crystal Hoyt\footnote{Nara Women's University, Japan; Weizmann Institute of Science, Israel; crystal.hoyt@weizmann.ac.il}\\ (joint work with Vera Serganova)}
\date{November 17, 2009}

\maketitle

\begin{abstract}
A contragredient Lie superalgebra is a superalgebra defined by a Cartan matrix.  In general, a contragredient Lie superalgebra is not finite dimensional, however it has a natural $\mathbb{Z}$-grading by finite dimensional components.
A contragredient Lie superalgebra has finite growth if the dimensions of these graded components depend polynomially on the degree.  We discuss the classification of finite-growth contragredient Lie superalgebras.
\end{abstract}

\setcounter{section}{-1}
\section{Introduction}\ \indent
A Lie superalgebra is a generalization of a Lie algebra, and a contragredient Lie superalgebra is a superalgebra defined by a Cartan matrix \cite{K77, K90}.  In 1977, V.G. Kac classified the simple finite-dimensional contragredient Lie superalgebras~\cite{K77}. In 1978,  V.G. Kac classified the finite-growth contragredient Lie superalgebras which satisfy the condition that there are no zeros on the main diagonal of the Cartan matrix~\cite{K78}.  In 1986, J.W. van de Leur classified the finite-growth contragredient Lie superalgebras which have symmetrizable Cartan matrices \cite{L86,L89}.  We complete the classification of finite-growth contragredient Lie superalgebras in \cite{HS07}, imposing no conditions.  Our list contains previously known examples \cite{KL89, S84} and some new superalgebras, however these new superalgebras are not simple.

{\def\thedefinition{\ref{171}} \addtocounter{definition}{-1}
\begin{theorem}[Hoyt, Serganova \cite{HS07}]  Let $\mathfrak{ g} (A)$ be a contragredient Lie superalgebra of finite growth and suppose the matrix $A$ is indecomposable with no zero rows.  Then either $A$ is symmetrizable and $\mathfrak{g}(A)$ is isomorphic to an affine or finite-dimensional Lie superalgebra classified in \cite{L86,L89},  or it is $D(2,1,0)$, $\widehat{D}(2,1,0)$, $S(1,2,\alpha)$ or $q(n)^{(2)}$.  See Table 1 and Table 3.
\end{theorem}}

The case where the matrix $A$ has a zero row is also handled in \cite{HS07}. In this case, the Lie superalgebra $\mathfrak{g}(A)$ is not simple and is basically obtained by extending a finite dimensional algebra by a Heisenberg algebra. A characterization of these Cartan matrices is given in \cite{HS07}.

Unlike the usual Kac-Moody algebra situation, there exist non-symmetrizable families of finite-growth contragredient Lie superalgebras, and one of these families, $S(1,2,\alpha)$, can not be realized as a twisted affinization, whereas another family, $q(n)^{(2)}$, is a twisted affinization of a finite dimensional Lie superalgebra which is not contragredient.  The Lie superalgebra $S(1,2,\alpha)$ appears in the list of conformal superalgebras given in \cite{KL89}.

The main idea of this classification is to use odd reflections  \cite{LSS86} to relate the properties of different bases for the Lie superalgebra g(A). In contrast to the Lie algebra case where the Cartan matrix is unique, a Lie superalgebra usually has more than one Cartan matrix.

It is shown in \cite{HS07} that if $\mathfrak{g}(A)$ is a finite-growth contragredient Lie superalgebra and the Cartan matrix $A$ has no zero rows, then simple root vectors of $\mathfrak{g}(A)$ act locally nilpotently on the adjoint module.  This implies certain conditions on $A$ which are only slightly weaker than the conditions for the matrix to be a generalized Cartan matrix. The crucial point is that for a finite-growth Lie superalgebra, these matrix conditions should still hold after odd reflections.

This leads to the definition of a regular Kac-Moody superalgebra: a contragredient Lie superalgebra all of whose Cartan matrices obtained by odd reflections are generalized Cartan matrices.  Equivalently, these are the contragredient Lie superalgebras for which all real root vectors act locally nilpotently on the adjoint module and all real isotropic roots are regular.  Remarkably, this is a finite list of families.

By comparing the classification of regular Kac-Moody superalgebras to the classification of
symmetrizable finite-growth contragredient Lie superalgebras in \cite{L86,L89} we obtain the following formulation of the classification theorem.

{\def\thedefinition{\ref{rkmtheorem}}\addtocounter{definition}{-1}
\begin{theorem}[Hoyt \cite{H07, H08}]
If $A$ is a symmetrizable indecomposable matrix and the contragredient Lie superalgebra $\mathfrak{g}(A)$ has a simple
isotropic root, then $\mathfrak{g}(A)$ is a regular Kac-Moody superalgebra if and only if it has finite growth.

If $A$ is a non-symmetrizable indecomposable matrix and the contragredient Lie superalgebra $\mathfrak{g}(A)$ has a
simple isotropic root, then $\mathfrak{g}(A)$ is a regular Kac-Moody superalgebra if and only if it belongs to one of the following three families: $q(n)^{(2)}$, $S(1,2,\alpha)$ with $\alpha\in\mathbb{C}\setminus\mathbb{Z}$, $Q^{\pm}(m,n,t)$ with $m,n,t\in\mathbb{Z}_{\leq -1}$.
\end{theorem}}

We discover that regular Kac-Moody superalgebras almost always have finite growth.  The only exception is a family of $3\times 3$-matrices, $Q^{\pm}(m,n,t)$.  This superalgebra is new, and a realization is not known.

We classify regular Kac-Moody superalgebras by extending finite type diagrams.  We then prove that our list is complete by using the theory of integrable highest weight modules of Kac-Moody superalgebras developed in \cite{KW01}.

\section{Lie superalgebras}\ \indent
A {\em superalgebra}\index{superalgebra} is a $\mathbb{Z}_{2}$-graded algebra $A=A_{\overline{0}} \oplus
A_{\overline{1}}$  over $\mathbb{C}$.  A {\em Lie superalgebra} is a
(non-associative) superalgebra $\mathfrak{g}=\mathfrak{g}_{\overline{0}} \oplus \mathfrak{g}_{\overline{1}}$ where the
product operation $[\cdot,\cdot]: \mathfrak{g} \times \mathfrak{g} \rightarrow \mathfrak{g}$ satisfies the following
axioms, with $x \in \mathfrak{g}_{p(x)}$, $y \in \mathfrak{g}_{p(y)}$:
\begin{align*}
&[x,y]=-(-1)^{p(x)p(y)} [y,x]  &\text{super skew-symmetry,} \\
&[x,[y,z]] = [[x,y],z] + (-1)^{p(x)p(y)} [y,[x,z]]  &\text{super Jacobi identity.}
\end{align*}
The subalgebra $\mathfrak{g}_{\overline{0}}$ is a Lie algebra.

\begin{example}
The general Lie superalgebra $\mathfrak{gl}(m|n)$:  Let $M_{r,s}$ denote the set of $r\times s$ matrices. As a vector space $\mathfrak{gl}(m|n)$ is $M_{m+n,m+n}$, where:
\begin{equation*}
\mathfrak{g}_{\overline{0}}=\left\{\left(\begin{array}{cc} A & 0 \\ 0 & B\end{array}\right)\mid A\in M_{m,m},\ B\in M_{n,n}\right\}, \text{ and }
\mathfrak{g}_{\overline{1}}=\left\{\left(\begin{array}{cc} 0 & C \\ D & 0\end{array}\right)\mid C\in M_{m,n},\ D\in M_{n,m}\right\}.
\end{equation*}
The bracket operation is defined on homogeneous elements as follows:  if $X\in\mathfrak{g}_{i}$, $Y\in\mathfrak{g}_{j}$, then $$[X,Y]:=XY-(-1)^{ij}YX,$$ and is extended linearly to the superalgebra.
\end{example}

\begin{example}
The special Lie superalgebra $$\mathfrak{sl}(m|n):=\left\{X=\left(\begin{array}{cc} A & C \\ D & B\end{array}\right)\in\mathfrak{gl}(m|n)\mid \mbox{supertr}(X):=\mbox{tr}(A)-\mbox{tr}(B)=0 \right\}.$$
\end{example}

\section{Contragredient Lie superalgebras}\ \indent
Let $A$ be a $n\times n$ matrix over $\mathbb{C}$, $I=\{1,\ldots,n\}$ and $p:I \to {\mathbb Z}_{2} $ be a parity function.  Fix a vector space $\mathfrak{h}$ over $\mathbb{C}$ of dimension $2n-\operatorname{rk}(A)$.
Let $\alpha_{1},\dots ,\alpha_{n}\in{\mathfrak h}^{*}$ and $h_{1},\dots ,h_{n}\in{\mathfrak h}$ be linearly independent elements satisfying $\alpha_{j}\left(h_{i}\right)=a_{ij}$, where $a_{ij}$ is the $ij$-th entry of $A$.

Define a Lie superalgebra $\tilde{\mathfrak{g}}(A)$ by generators $X_{1},\dots ,X_{n},Y_{1},\dots ,Y_{n}$ and $\mathfrak{h}$, and by relations
$$\left[X_{i},Y_{j}\right]=\delta_{ij}h_{i}\text{,\hspace{.5cm}
}\left[h,X_{i}\right]=\alpha_{i}\left(h\right)X_{i}\text{,\hspace{.5cm} }\left[h,Y_{i}\right]=-\alpha_{i}\left(h\right)Y_{i},\hspace{.5cm}\text{for all }h\in\mathfrak{h},$$
where the parity of $X_i$ and $Y_i$ is $p(i)$, and the elements of $\mathfrak{h}$ are even.

\newpage
Let $r(A)$ be the maximal ideal of $\tilde{\mathfrak{g}}(A)$ which intersects $\mathfrak{h}$ trivially.
Then $$\frak{g}(A):=\tilde{\mathfrak{g}}(A)/ r(A)$$ is a {\em contragredient Lie superalgebra}.

The matrix $A$ is the {\em Cartan matrix} of $\mathfrak{g}(A)$ for the set of simple roots $\Pi:=\{\alpha_1,\ldots,\alpha_n\}$.  If $B=DA$ for some invertible diagonal matrix $D$, then $\mathfrak{g}(B)\cong\mathfrak{g}(A)$.  Hence, we may assume without loss of generality that $a_{ii}\in\{0,2\}$ for $i\in I$.  The matrix $A$ is said to be {\em symmetrizable} if there exists an invertible diagonal matrix $D$ such that $B=DA$ is a symmetric matrix, i.e. $b_{ij}=b_{ji}$ for all $i,j\in I$.  In this case, we also say that $\mathfrak{g}(A)$ is symmetrizable.

We have the following lemma \cite{HS07}.

\begin{lemma} \label{lm20} For any subset $ J\subset I $ the subalgebra $\mathfrak{a}_{J}$ in $\mathfrak{g}\left(A\right) $ generated by $ \mathfrak{h} $, $ X_{i} $ and $ Y_{i} $, with $ i\in J $, is isomorphic to $ \mathfrak{h}'\oplus{\mathfrak g}\left(A_{J}\right) $, where $ A_{J} $ is the submatrix of $A$
with coefficients $ \left(a_{ij}\right)_{i,j\in J} $ and $\mathfrak{h}' $ is a subspace of $\mathfrak{h} $. More
precisely, $\mathfrak{h}' $ is a maximal subspace in $ \cap_{i\in J} \operatorname{Ker} \alpha_{i} $ which trivially
intersects the span of $ h_{i} $, $ i\in J $.
\end{lemma}

Let $\mathfrak{n}_{+}$ (resp. $\mathfrak{n}_{-}$) denote the subalgebra of $\mathfrak{g}(A)$ generated by the elements
$X_i$ (resp. $Y_i$), $i\in I$.  Then $\mathfrak{g}(A)=\mathfrak{n}^{-}\oplus\mathfrak{h}\oplus\mathfrak{n}^{+}$.

\section{Roots}\ \indent
The Lie superalgebra $ {\mathfrak g}={\mathfrak g}\left(A\right) $ has a {\em root space decomposition}
\begin{equation}
{\mathfrak g}={\mathfrak h}\oplus\bigoplus_{\alpha\in\Delta}{\mathfrak g}_{\alpha}.
\notag\end{equation}
Every root is either a positive or a negative linear combination of the simple roots, $\alpha_1,\ldots,\alpha_n$.  Accordingly, we  have the decomposition $ \Delta=\Delta^{+}\cup\Delta^{-} $, and we call $\alpha\in\Delta^{+}$ positive and $\alpha\in\Delta^{-}$ negative.
One can define $ p:\Delta \to {\mathbb Z}_{2} $ by letting $ p\left(\alpha\right)=0 $ or 1 whenever $\alpha $ is even or odd, respectively. By $ \Delta_{0}$ (resp. $\Delta_{1}$) we denote the set of even (resp. odd) roots.\\

There are four possibilities for each simple root:
\begin{enumerate}
\item if $ a_{ii}=2 $ and $ p(\alpha_i)=0 $,
then $ X_{i} $, $ Y_{i} $ and $ h_{i} $ generate a subalgebra isomorphic to $ \mathfrak{sl}\left(2\right) $;

\item if $ a_{ii}=0 $ and $ p(\alpha_i)=0 $,
then $ X_{i} $, $ Y_{i} $ and $ h_{i} $ generate a subalgebra isomorphic to the Heisenberg algebra;

\item if $ a_{ii}=2 $ and $ p(\alpha_i)=1 $, then $ X_{i} $, $ Y_{i} $ and $ h_{i} $ generate a subalgebra isomorphic to $ \mathfrak{osp}\left(1|2\right) $, and in this case $2\alpha_i\in\Delta $;

\item if $ a_{ii}=0 $ and $ p(\alpha_i)=1 $, then  $ X_{i} $, $ Y_{i} $ and $
h_{i} $ generate a subalgebra isomorphic to $ \mathfrak{sl}\left(1|1\right) $.
\end{enumerate}
In the last case we say that $ \alpha $ is {\em isotropic}, and in the other cases a root is called {\em non-isotropic}.  A simple root $\alpha_i$ is {\em regular} if for any other simple root $\alpha_j$, $a_{ij}=0$ implies $a_{ji}=0$.  Otherwise a simple root is called {\em singular}.\\

It is useful to describe $A$ by the corresponding Dynkin diagram, which we denote $\Gamma_{A}$, see \cite{K77,L86}. Here we use matrix diagrams, but still follow the standard labeling conventions for the vertices.

\begin{equation*}\doublespacing
\begin{tabular}{|c|c|c|c|}
\hline
$\mathfrak{g}(A)$ & $A$ & $p(1)$ & Dynkin diagram \\
\hline
$\mathfrak{sl}(2)$ & (2) & 0 & $\O$ \\
$\mathfrak{osp}(1,2)$ & (2) & 1 & \B \\
$\mathfrak{sl}(1,1)$ & (0) & 1 & $\OX$ \\
\hline
\end{tabular}
\end{equation*}
We join vertex $i$ to vertex $j$ by an arrow if $a_{ij} \neq 0$ and we write the number $a_{ij}$ on this arrow.

\newpage
\section{Odd reflections}\ \indent
Let $\Pi=\{\alpha_1,\ldots,\alpha_n\}$ be the set of simple roots of $\mathfrak{g}(A)$.  If $\alpha_k\in\Pi$ and $a_{kk}\neq 0$, we define the (even) reflection $r_{k}$ at $\alpha_k$ by
\begin{equation*}
r_{k}(\alpha_i)=\alpha_i - \alpha_i(h_k)\alpha_k,\hspace{2cm}\alpha_i\in\Pi.
\end{equation*}

\noindent
If $\alpha_k\in\Pi$, $a_{kk}=0$ and $p(\alpha_k)=1$, we define the {\em odd reflection} $r_{k}$ at $\alpha_k$ as follows:
\begin{equation*} r_{k}(\alpha_{i}):=  \left\{
  \begin{array}{ll}
    -\alpha_{k}, & \hbox{if $i=k$;} \\
    \alpha_{i}, & \hbox{if $a_{ik}=a_{ki}=0$, $i \neq k$;} \\
    \alpha_{i}+\alpha_{k}, & \hbox{if $a_{ik}\neq 0$ or $a_{ki} \neq 0$,  $i \neq k$;}
  \end{array}\hspace{1cm}\alpha_i\in\Pi,
\right.\end{equation*}

\begin{equation*}
 X_{i}':=\left\{
  \begin{array}{ll}
     Y_{i}, & \hbox{if $i=k$;}\\
     X_{i}, & \hbox{if $i\neq k$, and $a_{ik}=a_{ki}=0$;}\\
    {[X_{i},X_{k}]}, & \hbox{if $i\neq k$, and $a_{ik}\neq 0$ or $a_{ki}\neq 0$;}
\end{array}\right.
\end{equation*}
\begin{equation*}
 Y_{i}':=\left\{
  \begin{array}{ll}
     X_{i}, & \hbox{if $i=k$;}\\
     Y_{i}, & \hbox{if $i\neq k$, and $a_{ik}=a_{ki}=0$;}\\
    {[Y_{i},Y_{k}]}, & \hbox{if $i\neq k$, and $a_{ik}\neq 0$ or $a_{ki}\neq 0$;}
\end{array}\right.
\end{equation*}
and
\begin{equation*}
h_{i}' := [X_{i}',Y_{i}'].
\end{equation*}
Then $$ h_{i}' = \begin{cases}
    h_{k} &\text{ if } i=k,\\
    h_{i} &\text{ if }i \neq k\text{, and }a_{ik}=a_{ki}=0, \\
    (-1)^{p(\alpha_{i})} (a_{ik}h_{k}+a_{ki}h_{i}) &\text{ if }i \neq k\text{, and }a_{ik} \text{ or }a_{ki} \neq 0. \end{cases} $$
Set $\alpha'_i:=r_{k}(\alpha_i)$ for $i\in I$.\\

We have the following lemma \cite{HS07}.

\begin{lemma} The roots $\alpha'_1,\dots\alpha'_n$ are linearly independent.  The elements
\begin{equation*}
X'_1,\dots,X'_n \text{, } \hspace{.5cm} Y'_1,\dots,Y'_n \text{, } \hspace{.5cm} h'_1,\dots,h'_n
\end{equation*}
satisfy the Chevalley relations
\begin{equation*}
\left[h'_j,X'_{i}\right]=\alpha'_{i}\left(h'_j\right)X'_{i}
\text{, } \hspace{.5cm} \left[h'_j,Y'_{i}\right]=-\alpha'_{i}\left(h'_j\right)Y'_{i}
\text{, } \hspace{.5cm} \left[X'_{i},Y'_{j}\right]=\delta_{ij}h'_{i}.
\end{equation*}
Moreover, if $\alpha_k$ is regular, then $X'_i, Y'_i, \mathfrak{h}$ generate $\mathfrak{g}(A)$.
\end{lemma}

Given a matrix $A$ and a regular isotropic root $\alpha_k$, one can construct a new matrix $A'$ such that $\mathfrak{g}(A')$ and $\mathfrak{g}(A)$ are isomorphic as Lie superalgebras.  Explicitly, the entries of $A'$ are defined by $$a'_{ij}:=\alpha'_j(h'_i).$$  After rescaling the rows of $A'$ we obtain the following (where we assume $ i \neq k$ and $ j \neq k $):
\begin{equation*}
    a_{kk}' := a_{kk}; \hspace{1cm}
    a_{kj}' := a_{kj}; \hspace{1cm}
    a_{ik}' := -a_{ki}a_{ik}; \end{equation*}\begin{equation*}
    a_{ij}' := \begin{cases}
        a_{ij}, &\text{ if } a_{ik}=a_{ki}=0; \\
        a_{ki}a_{ij}, &\text{ if }a_{ik} \text{ or } a_{ki} \neq 0, \text{ and } a_{kj}=a_{jk}=0; \\
        a_{ki}a_{ij}+a_{ik}a_{kj}+a_{ki}a_{ik}, &\text{ if }a_{ik} \text{ or } a_{ki}\neq 0, \text{ and } a_{jk} \text{ or } a_{kj} \neq 0.
        \end{cases}
\end{equation*}
We say that $A'$ is obtained from $A$ (and $\Pi':=\{\alpha'_1,\ldots,\alpha'_n\}$ is obtained from $\Pi$) by an odd reflection with respect to $\alpha_k$.  If $\Delta'^{+}$ is the set of positive roots with respect to $\Pi'$, then
\begin{equation*}
\Delta'^{+}=\left(\Delta^{+}\setminus\{\alpha_k\}\right)\cup\{-\alpha_k\}.
\end{equation*}
We call $\Pi'$ a {\em reflected base} if it can be obtained from $\Pi$ by a sequence of even and odd reflections A root $\alpha$ is called {\em real} if $\alpha$ or $\frac{1}{2} \alpha$ is simple in some reflected base $\Pi'$, and it is called {\em imaginary} otherwise.\\

An odd reflection at a regular isotropic root $\alpha_k$ is indeed a reflection.  If $A''=r'_k (r_k(A))$, then there is an invertible diagonal matrix $D$ such that $A''=DA$ and scalars $b_i,c_i$ such that $X''_i=b_i X_i$ and $Y''_i=c_i Y_i$.  However, if $\alpha_k$ is singular, then the subalgebra generated by $ X'_{1},\dots ,X'_{n},Y'_{1},\dots ,Y'_{n} $ and $ {\mathfrak h} $ is necessarily a proper subalgebra of $\mathfrak{g}(A)$ and $r_k$ is not invertible.

\begin{example}  $\mathfrak{sl}(1|2)$
\begin{equation*}
\left(\begin{array}{ccccc}
0 & 1 \\
1 & 0 \\
\end{array}\right)
\hspace{1cm} \overrightarrow{ r_1 } \hspace{1cm}
\left(\begin{array}{cc}
0 & 1 \\
-1 & 2\\
\end{array}\right)
\end{equation*}
\begin{equation*}
\xymatrix{ \OX \AW[r]^{1}_{1} & \OX }
 \hspace{1cm} \overrightarrow{ r_1 } \hspace{1cm}
\xymatrix{\OX \AW[r]^{1}_{-1} & \O }
\end{equation*}
\end{example}

The following definition was given by V. Serganova in \cite{S08}.

\begin{definition}
Let $C$  be a category with objects $\mathfrak{g}(A)$ for each matrix $A$.  A $C$-morphism $f:\mathfrak{g}(A)\rightarrow\mathfrak{g}(A')$ is by definition an isomorphism of superalgebras which maps a Cartan subalgebra of $\mathfrak{g}(A)$ to a Cartan subalgebra of $\mathfrak{g}(A')$.  A Weyl groupoid $C(A)$ of a contragredient Lie superalgebra $\mathfrak{g}(A)$ is by definition the connected component of $C$ which contains $\mathfrak{g}(A)$.
\end{definition}

\section{Integrability conditions}\ \indent
We say that $ {\mathfrak g}\left(A\right) $ is {\em integrable} if $ \operatorname{ad}_{X_{i}}
$ are locally nilpotent for all $ i\in I $. In this case $ \operatorname{ad}_{Y_{i}} $ are also locally nilpotent.
A matrix $A$ is {\em indecomposable} if the the index set $I$ can not be decomposed into the disjoint union of non-empty subsets $J,K$ such that $a_{j,k}=a_{k,j}=0$ whenever $j\in J$ and $k\in K$.

\begin{lemma} \label{lm23}   Let $ A $ be be an indecomposable with $ n\geq2 $ and no zero rows. If $ {\mathfrak g}\left(A\right) $ is integrable, then after rescaling the rows $ A $ satisfies the following
conditions:
\begin{enumerate}
\item for any $ i\in I $ either $ a_{i i}=0 $ or $ a_{i i}=2 $;
\item if $ a_{i i}=0 $ then $ p\left(i\right)=1 $;
\item if $ a_{ii}=2 $ then $ a_{ij}\in2^{p\left(i\right)}{\mathbb Z}_{\leq 0} $;
\item if $ a_{ij}=0 $ and $ a_{ji}\not=0 $, then $ a_{i i}=0 $.
\end{enumerate}
\end{lemma}

A matrix $A$ is called a {\em generalized Cartan matrix} if it satisfies conditions 1-3, together with
\begin{equation*} 4'.\ \text{ if } a_{ij}=0\text{, then }a_{ji}=0. \end{equation*}
Condition $4'$ is equivalent to the condition that $\mathfrak{g}(A)$ has no real singular isotropic roots.

\section{Finite growth}\ \indent
A superalgebra $\mathfrak{g}=\mathfrak{g}(A)$ has a natural $\mathbb{Z}$-grading $\mathfrak{g}=\oplus\mathfrak{g}_{m}$, called the {\em principal grading}, which is defined by  $\mathfrak{g}_{0}:=\mathfrak{h} $ and $\mathfrak{g}_{1}:=\mathfrak{g}_{\alpha_{1}}\oplus\dots \oplus\mathfrak{g}_{\alpha_{n}} $, where $\Pi=\{\alpha_1,\ldots,\alpha_n\}$.  We say that ${\mathfrak g} $ is of {\em finite growth} if $ \dim\mathfrak{g}_{m} $ grows polynomially depending on $m$. This means that the Gelfand-Kirillov dimension of $\mathfrak{g}$ is finite.

The following lemma shows that the notion of finite growth does not depend on the choice of a reflected base \cite{HS07}.

\begin{lemma} \label{lm2}  Let $ \Pi' $ be obtained from $ \Pi $ by an odd reflection. Then $ \dim
{\mathfrak g}'_{m}\leq\dim  {\mathfrak g}_{-2m}\oplus\dots \oplus{\mathfrak g}_{2m} $. In particular, if $ {\mathfrak
g}\left(A\right) $ is of finite growth, then $ {\mathfrak g}\left(A'\right) $ is of finite growth.
\end{lemma}

The following lemma is a straightforward corollary of Lemma~\ref{lm20}.

\begin{lemma} \label{lm21}  If $ {\mathfrak g}\left(A\right) $ is of finite growth,
then for any subset $ J \subseteq I $ the Lie superalgebra $ {\mathfrak g}\left(A_{J}\right) $ is of finite growth.

\end{lemma}

\newpage

\begin{theorem}[Hoyt, Serganova \cite{HS07}] \label{th1}  Suppose $A$ is a matrix with no zero rows. If $ {\mathfrak g}\left(A\right) $ has finite growth, the $\mathfrak{g}(A)$ is integrable.
\end{theorem}

\begin{theorem}[Hoyt, Serganova \cite{HS07}]\label{admiss}  Suppose that $A$ is an indecomposable matrix with no zero rows, and that $ {\mathfrak g}\left(A\right) $ has finite growth. Then $A$ and each matrix $A'$ obtained from $A$ by a sequence of odd reflections (singular or regular) satisfies the conditions of Lemma~\ref{lm23}, (after rescaling so that $a_{ii}\in\{0,2\}$ for $i\in I$).
\end{theorem}
\begin{proof}
Suppose that $A$ has no zero rows, and let $A'$ be a matrix obtained from $A$ by a sequence of odd reflections.  It is easy to show that $A'$ has no zero rows.  Suppose $ {\mathfrak g}\left(A\right) $ has finite growth.  Then by Lemma~\ref{lm2}, $\mathfrak{g}(A')$ has finite growth.  Then by Theorem~\ref{th1}, $\mathfrak{g}(A)$ and $\mathfrak{g}(A')$ are integrable. Hence, by Lemma~\ref{lm23}, $A$ and $A'$ satisfy the matrix conditions of Lemma~\ref{lm23}.
\end{proof}

We call $A$ {\em admissible} if any matrix obtained from $A$ by a sequence of odd reflections satisfies the conditions of Lemma~\ref{lm23}.  By Theorem~\ref{admiss}, if $A$ is an indecomposable matrix with no zero rows and $ {\mathfrak g}\left(A\right) $ has finite growth, then $A$ is admissible.  Hence, it suffices to only consider admissible Cartan matrices in the classification of finite-growth contragredient Lie superalgebras.

\section{Main Theorem}
\begin{theorem}[Hoyt, Serganova \cite{HS07}] \label{171} Let $\mathfrak{ g} (A)$ be a contragredient Lie
superalgebra of finite growth, and suppose the matrix $A$ is indecomposable with no zero rows. Then either $A$ is
symmetrizable and $\mathfrak{g}(A)$ is isomorphic to an affine or finite dimensional Lie superalgebra classified in
\cite{L86,L89}, or it is $D(2,1,0)$, $\widehat{D}(2,1,0)$, $S(1,2,\alpha)$ or $q(n)^{(2)}$.  See Table 1 and Table 3.
\end{theorem}

The proof of this theorem is separated into two cases, namely (1) $A$ is admissible and $\mathfrak{g}(A)$ has no singular isotropic roots in any reflected base, and (2) $A$ is admissible and $\mathfrak{g}(A)$ has a singular isotropic root.  In the first case, $\mathfrak{g}(A)$ is a {\em regular Kac-Moody superalgebra}.

\section{Regular Kac-Moody superalgebras}\label{rkmdef}\ \indent
We call $\mathfrak{g}(A)$ a {\em regular Kac-Moody superalgebra} if $A$ and any matrix obtained from $A$ by a sequence of odd reflections is a generalized Cartan matrix.  Equivalently, these are the contragredient Lie superalgebras for which all real root vectors act locally nilpotently on the adjoint module and all real isotropic roots are regular.

We classify regular Kac-Moody superalgebras by classifying the corresponding Dynkin diagrams. We proceed by induction on the number of vertices.  If $\mathfrak{g}(A)$ does not have a simple isotropic root, then $\mathfrak{g}(A)$ is a regular Kac-Moody superalgebra if and only if $A$ is a generalized Cartan matrix.  So we may assume that $\mathfrak{g}(A)$ has a simple isotropic root. If $\mathfrak{g}(A)$ has finite type, then $\mathfrak{g}(A)$ is a regular Kac-Moody superalgebra.  First, we classify regular Kac-Moody diagrams $\Gamma_{A}$ which satisfy the additional condition:
\begin{equation}\label{gammaproperty}
\text{ for } \Gamma_{A} \text{ and any reflected diagram } \Gamma_{A'} \text{ all proper subdiagrams have finite type.}
 \end{equation}
Second, we prove that our list is complete using Corollary~\ref{lm123} and some direct computations.

\begin{theorem}[Hoyt, Serganova \cite{HS07}]\label{two}
The regular Kac-Moody superalgebras with two simple roots of which at least one is isotropic are $\mathfrak{sl}(1,2)$ and $\mathfrak{osp}(3,2)$.
\end{theorem}

\begin{proof}
If $a\neq -1$, then by reflecting at $v_{1}$, we have the conditions $a, \frac{-a}{a+1} \in {\mathbb Z}_{<0}$. This implies $a=-2$.  Hence, this is a Dynkin diagram for $\mathfrak{osp}(3,2)$.
\begin{equation*}
    \xymatrix{ \OX \AW[r]^{1}_{a} & \OD}
    \text{\hspace{.5cm}} \overrightarrow{r_{1}} \text{\hspace{.5cm}}
    \xymatrix{ \OX \AW[r]^{1}_{- \frac{a}{1+a}} & \OD}
\end{equation*}
If $a=-1$, then this is a Dynkin diagram for $\mathfrak{sl}(1,2)$.  By reflecting at $v_{1}$ we have the following.
\begin{equation*}
    \xymatrix{ \OX \ar@<.5ex>[r]^{1} & \O \ar@<.5ex>[l]^{-1}}
    \text{\hspace{.5cm}} \overrightarrow{r_{1}} \text{\hspace{.5cm}}
    \xymatrix{ \OX \ar@<.5ex>[r]^{1} & \OX \ar@<.5ex>[l]^{1}}
\end{equation*}
\end{proof}

\newpage
\section{Classification of regular Kac-Moody superalgebras}

\begin{theorem}[Hoyt \cite{H07, H08}]\label{rkmtheorem} Let $\mathfrak{g}(A)$ be a contragredient Lie superalgebra.
\begin{enumerate}
\item
If $A$ is a symmetrizable matrix and $\mathfrak{g}(A)$ has a simple isotropic root, then $\mathfrak{g}(A)$ is regular Kac-Moody if and only if it has finite growth.
\item
If $A$ is a non-symmetrizable matrix and $\mathfrak{g}(A)$ has a simple isotropic root, then $\mathfrak{g}(A)$ is regular
Kac-Moody if and only if it is  $q(n)^{(2)}$, $S(1,2,\alpha)$ or $Q^{\pm}(m,n,t)$.  See Table 1 and Table 2.\\
\end{enumerate}
\begin{equation*}\doublespacing \begin{tabular}{|c|cc|}
  \hline
  \bf{Algebra} & \bf{Dynkin diagrams} & Table 1 \\
\hline &&\\
$q(n)^{(2)}$ & $\begin{array}{c}\xymatrix{ & & \bullet \ar@<.5ex>[lld]^{a} \ar@<.5ex>[rrd]^{b} & & \\
\bullet \ar@<.5ex>[rru] \ar[r] & \bullet \ar[l] \ar[r] & \cdots \ar[l] \ar[r] & \bullet \ar[l] \ar[r]
& \bullet \ar[l] \ar@<.5ex>[llu] }\end{array}$ & $\begin{array}{c} \text{There are n $\bullet$.} \\
\text{Each $\bullet$ is either $\bigcirc$ or $\bigotimes$.} \\ \text{An odd number of them are
$\bigotimes$.} \\
\text{If $\bullet$ is $\bigcirc$,
 then $a=b=-1$.} \\
\text{If $\bullet$ is $\bigotimes$, then $\frac{a}{b}=-1$.} \\
 \end{array}$
\\ &&\\
\hline &&\\
$S(1,2,\alpha)$ &
$\begin{array}{c}\xymatrix{& \bigotimes \ar@<.5ex>[rdd]^{1-\alpha} \ar@<.5ex>[ldd]^{\alpha} & \\
& & \\ \bigotimes \ar@<.5ex>[ruu]^{\alpha} \ar@<.5ex>[rr]^{-1-\alpha} & &
\bigcirc \ar@<.5ex>[ll]^{-1} \ar@<.5ex>[luu]^{-1}}\end{array}$ & $\alpha \neq 0, \text{  }\alpha \in \mathbb{C} \setminus \mathbb{Z}$\\ &&\\
 \hline
 \end{tabular}\end{equation*}\

 \begin{equation*}\doublespacing \begin{tabular}{|c|cc|}
  \hline
  \bf{Algebra} & \bf{Dynkin diagrams} & Table 2  \\
\hline &&\\
$Q^{\pm}(m,n,t)$ & $\begin{array}{c}\xymatrix{& \bigotimes \ar@<.5ex>[rdd]^{b} \ar@<.5ex>[ldd]^{1} & \\
& & \\ \bigotimes \ar@<.5ex>[ruu]^{c} \ar@<.5ex>[rr]^{1} & & \bigotimes \ar@<.5ex>[ll]^{a} \ar@<.5ex>[luu]^{1}}\end{array}$ \text{
}
$\begin{array}{c}\xymatrix{& \bigotimes \ar@<.5ex>[rdd]^{b} \ar@<.5ex>[ldd]^{1} & \\
& & \\ \bigcirc \ar@<.5ex>[ruu]^{-1} \ar@<.5ex>[rr]^{1+b+\frac{1}{c}} & &
\bigcirc \ar@<.5ex>[ll]^{1+a+\frac{1}{b}} \ar@<.5ex>[luu]^{-1}}\end{array}$
& $\begin{array}{l}
1+a+\frac{1}{b} = m \\
1+b+\frac{1}{c} = n \\
1+c+\frac{1}{a} = t \\
\end{array}$ \\
 & $\begin{array}{c}\xymatrix{& \bigotimes \ar@<.5ex>[rdd]^{c} \ar@<.5ex>[ldd]^{1} & \\
& & \\ \bigcirc \ar@<.5ex>[ruu]^{-1} \ar@<.5ex>[rr]^{1+c+\frac{1}{a}} & & \bigcirc \ar@<.5ex>[ll]^{1+b+\frac{1}{c}}
\ar@<.5ex>[luu]^{-1}}\end{array}$ \text{  }
$\begin{array}{c}\xymatrix{& \bigotimes \ar@<.5ex>[rdd]^{a} \ar@<.5ex>[ldd]^{1} & \\
& & \\ \bigcirc \ar@<.5ex>[ruu]^{-1} \ar@<.5ex>[rr]^{1+a+\frac{1}{b}} & &
\bigcirc \ar@<.5ex>[ll]^{1+c+\frac{1}{a}} \ar@<.5ex>[luu]^{-1}}\end{array}$ &
 $\begin{array}{l}
 m,n,t \in \mathbb{Z}_{\leq -1} \text{ and}\\
 \text{not all equal to }-1. \\
 \end{array}$ \\ &&\\
\hline
\end{tabular}
\end{equation*}
\end{theorem}
\newpage

\section{Principal roots}\ \indent
We call a root $\alpha\in\Delta_{0}$ {\em principal} if either $\alpha$ or $\frac{1}{2}\alpha$ belongs to some base $\Pi'$ obtained from $\Pi$ be a sequence of regular odd reflections. For a principal root, the subalgebra generated by $X_{\alpha}$, $Y_{\alpha}$ and $h_{\alpha}:=[X_{\alpha},Y_{\alpha}]$ is isomorphic to $\mathfrak{sl}_2$, and we may choose $X_{\alpha}$, $Y_{\alpha}$ such that $\alpha(h_{\alpha})=2$.  Note that if $\alpha=2\beta$, then $X_{\alpha}=[X_{\beta},X_{\beta}]$.

Let $\Pi_{0}\subset\Delta_{0}$ denote the set of principal roots.  It is clear that $\Pi_{0}\subset\Delta_{0}^{+}$.   We can define the {\em Weyl group}\index{Weyl group} to be the group generated by reflections $r_{\alpha}$ for $\alpha\in\Pi_{0} $.   In general $\Pi_{0}$ can be infinite.  The results in this section are from \cite{HS07}.

\begin{lemma} \label{lm14}  If $ {\mathfrak g}\left(A\right) $ is a regular Kac-Moody superalgebra of finite growth, then $ \Pi_{0} $ is finite. Moreover, $ |\Pi_{0}|\leq2n $.
\end{lemma}

For any finite subset $S\subset \Pi_{0}$, we can define a matrix $B$ by setting $b_{ij}=\alpha_j(h_i)$.  The corollary of the following lemma is an important tool in the classification of finite-growth contragredient Lie superalgebras.

\begin{lemma} \label{lm4} Let $S$ be a
  subset of $\Pi_0$, and let ${\mathfrak g}_S$ be the subalgebra of ${\mathfrak g}$ generated by $
  X_{\beta},Y_{\beta} $ for all $\beta \in S$ and $ h_{\beta}
  =[X_{\beta}, Y_{\beta}]$.
These elements satisfy the following relations for the matrix $B$ defined by $S$:
$$[h_{\beta}, X_{\gamma}]=b_{\beta , \gamma} X_{\beta},
[h_{\beta}, Y_{\gamma}]=-b_{\beta , \gamma} Y_{\beta}, [X_{\beta}, Y_{\gamma}]={\delta}_{\beta , \gamma} h_{\beta},\text{
for } \beta,\gamma\in{S}.
$$
There exists a surjective homomorphism ${\mathfrak g}_S \to {\mathfrak
  g}'(B) /c$ where $c$ is some central subalgebra of the Cartan
subalgebra in ${\mathfrak g}' (B)=[{\mathfrak g}(B), {\mathfrak g} (B)]$.
\end{lemma}

\begin{corollary} \label{cor1}  If a Lie superalgebra $ {\mathfrak g}\left(A\right) $ has finite growth, then
for any finite subset $S$ of ${\Pi}_0$ the Lie algebra $ {\mathfrak g}\left(B\right) $ also has finite growth. In particular, $ B $ is an even generalized Cartan matrix.
\end{corollary}

\section{The Lie superalgebra $Q^{\pm}(m,n,t)$}\label{sectionQ}
\begin{equation}\label{Qroots}
\begin{array}{l}\xymatrix{& \OX \AW[ldd]^{1}_{c} \AW[rdd]^{b}_{1} & \\ & & \\
\OX \AW[rr]^{1}_{a} & & \OX }
\end{array}\
\begin{array}{l}
1+a+\frac{1}{b} = m \\
1+b+\frac{1}{c} = n \\
1+c+\frac{1}{a} = t \\ \\
m,n,t \in \mathbb{Z}_{\leq -1} \text{, not all equal to -1.} \end{array}
\end{equation}

\begin{theorem}[Hoyt \cite{H07, H08}]\label{lm333}
For each diagram $Q^{\pm}(m,n,t)$, it follows that $a,b,c \in \mathbb{R}\setminus\mathbb{Q}$, and there are two solutions of the above equations, namely $Q^{-}(m,n,t)$ with $a,b,c<-1$ and
$Q^{+}(m,n,t)$ with $-1<a,b,c<0$.
\end{theorem}

\begin{corollary}\label{cor3}
The determinant of the Cartan matrix equals $1 + abc$ and is nonzero. Hence, the dimension of the Cartan subalgebra is
$3$.
\end{corollary}

It is clear that $Q^{\pm}(m,n,t) \cong Q^{\pm}(n,t,m)$ by cyclic permutation of the variables $a,b,c$.  We also have
$Q^{\pm}(m,n,t) \cong Q^{\mp}(m,t,n)$ by transforming the equations: $a \rightarrow \frac{1}{b}$, $b \rightarrow
\frac{1}{a}$, $c \rightarrow \frac{1}{c}$.

\begin{lemma}[Hoyt \cite{H07, H08}] $Q^{\pm}(m,n,t)$ is not symmetrizable. \end{lemma}
\begin{proof}
Suppose that $Q^{\pm}(m,n,t)$ is symmetrizable, for some $m,n,t\in\mathbb{Z}_{\geq -1}$.  Then $abc = 1$.  A simple computation shows that this implies $a,b,c\in\mathbb{Q}$, which contradicts Theorem~\ref{lm333}.
\end{proof}

\begin{lemma}[Hoyt \cite{H07, H08}] $Q^{\pm}(m,n,t)$ does not have finite growth. \end{lemma}
\begin{proof}
Let $\Pi=\{\alpha_1,\alpha_2,\alpha_3\}$ be the set of simple roots of the Dynkin diagram in (\ref{Qroots}).
Then the set of principal roots is $\Pi_{0}=\{\alpha_1+\alpha_2,\alpha_2+\alpha_3,\alpha_3+\alpha_1\}$.
The matrix $B$ given by $\Pi_{0}$ is
\begin{equation*}
B = \left( \begin{array}{ccc}
2 & m & m  \\
n & 2 & n \\
t & t & 2 \\
\end{array} \right).
\end{equation*}
All off diagonal entries of $B$ are negative integers.  Since they are not all equal to $-1$, this is not a Cartan matrix of a finite growth Kac-Moody algebra.  By Corollary~\ref{cor1}, $Q^{\pm}(m,n,t)$ does not have finite growth.
\end{proof}

\section{The singular case}
\begin{theorem}[Hoyt, Serganova \cite{HS07}] Suppose that $A$ is an indecomposable matrix with no zero rows and that $\mathfrak{g}(A)$ has a singular isotropic simple root.  If $\mathfrak{g}(A)$ has finite growth, then $\mathfrak{g}(A)$ is isomorphic to $S(1,2,0)$, $D(2,1,0)$ or  $\widehat{D}(2,1,0)$.  See Table 3.
\end{theorem}

\begin{equation*}\doublespacing \begin{tabular}{|c|cc|}
  \hline
  \bf{Algebra} & \bf{Dynkin diagrams} & Table 3 \\
\hline  &&\\
$D(2,1,0)$ & $\xymatrix{\bigcirc \ar@<.5ex>[r]^{-1} & \bigotimes
\ar@<.5ex>[l]^{0} \ar@<.5ex>[r]^{1}
 & \bigcirc \ar@<.5ex>[l]^{-1} }$ &\\ &&\\
  \hline
$\widehat{D}(2,1,0)$ &
$\begin{array}{c}\xymatrix{
     & & \bigcirc  \ar@<.5ex>[ld]^{-1}  \\
     \bigcirc \ar@<.5ex>[r]^{-1} & \bigotimes \ar@<.5ex>[l]^{0} \ar@<.5ex>[ru]^{1}
     \ar@<.5ex>[rd]^{-1} & \\
     & & \bigcirc  \ar@<.5ex>[lu]^{-1}  }\end{array} $ & \\
    \hline
$S(1,2,0)$  &
$\begin{array}{c}\xymatrix{& \bigotimes \ar@<.5ex>[rdd]^{1-\alpha} \ar@<.5ex>[ldd]^{\alpha} & \\
& & \\ \bigotimes \ar@<.5ex>[ruu]^{\alpha} \ar@<.5ex>[rr]^{-1-\alpha} & &
\bigcirc \ar@<.5ex>[ll]^{-1} \ar@<.5ex>[luu]^{-1}}\end{array}$ & $\alpha \in \mathbb{Z}$ \\
\hline
\end{tabular}\end{equation*}

\section{Integrable modules}\label{rkmintegrable}\ \indent
Let $U(\mathfrak{g})$ denote the universal enveloping algebra of a Lie superalgebra $\mathfrak{g}$.  Corresponding to a set of simple roots $\Pi$ of $\mathfrak{g}$ we have the decomposition $\mathfrak{g}=\mathfrak{n}^{-}\oplus\mathfrak{h}\oplus\mathfrak{n}^{+}$.
A $\mathfrak{g}$-module $V$ is a {\em weight module} if $V=\oplus_{\mu \in \mathfrak{h}^{*}} V_{\mu}$, where
$V_{\mu}=\{v\in V \mid hv=\mu(h)v,\ \forall h\in\mathfrak{h}\}$ are finite dimensional vector spaces. If $V_{\mu}$ is non-zero, then $\mu$ is a {\em weight}.  A weight module $V$ is a {\em highest weight module} with {\em highest weight} $\lambda \in \mathfrak{h}^{*}$ if there exists a vector $v_{\lambda}\in V$ such that:
$\mathfrak{n}_{+}v_{\lambda}=0$, $hv=\lambda(h)v_{\lambda}$ for $h\in\mathfrak{h}$, and $U(\mathfrak{g})v_{\lambda}=V$.

A {\em Verma module} $$M(\lambda):=U(\mathfrak{g})\otimes_{U(\mathfrak{b}^{+})}C_{\lambda}$$ is an induced module, where   $C_{\lambda}$ is the one dimensional module over $\mathfrak{b}^{+}:=\mathfrak{h}\oplus\mathfrak{n}^{+}$ defined by $hv=\lambda(h)v$ for $v\in C_{\lambda}$ and for all $h\in\mathfrak{h}$, and $\mathfrak{n}^{+}$ acts trivially.
A Verma module is a highest weight module with highest weight $\lambda$.  Let $L(\lambda)$ denote the unique irreducible quotient of $M(\lambda)$.

Let $V$ be a weight module over $\mathfrak{g}$.  An element $x\in\mathfrak{g}$ acts {\em locally nilpotent} on $V$ if for any $v\in V$ there exists a positive integer $N$ such that $x^{N}v=0$.
Suppose $\mathfrak{g}$ is a regular Kac-Moody superalgebra.  We say that $V$ is an {\em integrable} module if for every real root $\alpha$  the element $X_{\alpha} \in \mathfrak{g}_{\alpha} $ acts locally nilpotent on $V$.

The results of this section are from \cite{H07, H08}, and are used in the classification of regular Kac-Moody superalgebras.

\begin{lemma}\label{lm411}
The adjoint module of a regular Kac-Moody superalgebra is an integrable module.
\end{lemma}

\begin{lemma}\label{M}
Suppose $\Gamma$ and $\Gamma' = \Gamma \cup \{v_{n+1}\}$ are connected regular Kac-Moody diagrams.  Let $\mathfrak{g}(A)$
(resp. $\mathfrak{g}(A')$) be the Kac-Moody superalgebra with diagram $\Gamma$ (resp. $\Gamma'$), and let $Y_{n+1} \in
\mathfrak{g}(A')_{-\alpha_{n+1}}$ be the generator corresponding to the vertex $v_{n+1}$. Then the submodule $M$ of
$\mathfrak{g}(A')$ generated by $\mathfrak{g}(A)$ acting on $Y_{n+1}$ is an integrable highest weight module over the
subalgebra $\mathfrak{g}(A)$.
\end{lemma}

\begin{proof}
The fact that $M$ is a highest weight module follows immediately from $[X_i,Y_{n+1}]=0$ for all $i=1,\dots,n$. The module
$M$ has highest weight $-\alpha_{n+1}$.  A real root $\alpha$ of the subalgebra $\mathfrak{g}(A)$ is also a real root of
$\mathfrak{g}(A')$.  By Lemma~\ref{lm411}, the adjoint module of $\mathfrak{g}(A')$ is integrable.  Thus for each real
root $\alpha$ of $\mathfrak{g}(A)$ we have that $Y_{\alpha} \in \mathfrak{g}(A)_{-\alpha}$  acts locally nilpotently on
the submodule $M$ of $\mathfrak{g}(A')$.  Hence the submodule $M$ is an integrable highest weight module over the
subalgebra $\mathfrak{g}(A)$.
\end{proof}

\begin{corollary}
Suppose $\Gamma$ is a diagram for a regular Kac-Moody superalgebra $\mathfrak{g}(A)$ which does not have non-trivial
irreducible integrable highest weight modules. Then $\Gamma$ is not a proper subdiagram of a connected regular Kac-Moody diagram.
\end{corollary}

\begin{proof}
Suppose  $\Gamma' = \Gamma \cup \{v_{n+1}\}$ is a connected regular Kac-Moody diagram.  By Lemma~\ref{M}, $M$ is an integrable highest weight module for $\mathfrak{g}(A)$ with irreducible quotient isomorphic to $L(-\alpha_{n+1})$. Since this module must be trivial, $-\alpha_{n+1}=0$.  This implies $a_{i,n+1}=0$ for $i=1,\dots,n$.  Since $\mathfrak{g}(A)$ is a regular Kac-Moody algebra, this implies $a_{n+1,j}=0$ for $j=1,\dots,n$. But this implies that the vertex $v_{n+1}$ is not connected to the diagram $\Gamma$, which is a contradiction. Since any subdiagram of a regular Kac-Moody diagram is a regular Kac-Moody diagram, the corollary follows.
\end{proof}

Our definition of integrable module coincides with the standard definition of integrable module in the case that
$\mathfrak{g}(A)$ is an affine Lie superalgebra.  This follows from the fact that every even root of an affine Lie
superalgebra with non-zero length is real, as was shown in \cite{S88}.
The {\em defect} of $\mathfrak{g}(A)$ is the dimension of a maximal isotropic
subspace of $\mathfrak{h}^{*}_{\mathbb{R}}:=\sum_{\alpha\in\Delta}\mathbb{R}\alpha$ \cite{KW94}.
It was shown in \cite{KW01}, that non-twisted affine Lie superalgebras with defect  strictly greater than one do not have non-trivial irreducible integrable highest weight modules.
We show that this also holds for twisted affine Lie superalgebras in \cite{H07, H08}.

\begin{theorem}[Kac, Wakimoto \cite{KW01}]\label{p1} The only non-twisted affine Lie superalgebras with non-trivial irreducible integrable highest weight
modules are $\mathfrak{osp}(1|2n)^{(1)}$, $\mathfrak{osp}(2|2n)^{(1)}$ and $\mathfrak{sl}(1|n)^{(1)}$.
\end{theorem}

\begin{theorem}[Hoyt \cite{H07, H08}]\label{p2}
Let $\mathfrak{g}^{(m)}$ be a twisted affine Lie superalgebra which is not
$\mathfrak{osp}(2|2n)^{(2)}$, $\mathfrak{sl}(1|2n)^{(2)}$ or $\mathfrak{sl}(1|2n+1)^{(4)}$.  Then an irreducible integrable highest weight module over $\mathfrak{g}^{(m)}$ is trivial.
\end{theorem}

\begin{corollary}\label{lm123}
Suppose $\mathfrak{g}(A)$ is a regular Kac-Moody superalgebra with a simple isotropic root such that $\Gamma_{A}$ satisfies condition (\ref{gammaproperty}).  If $\mathfrak{g}(A)$ is not of finite type and is not $\mathfrak{sl}(1|n)^{(1)}$, $\mathfrak{osp}(2|2n)^{(1)}$, $S(1,2,\alpha)$ or $Q^{\pm}(m,n,t)$, then all irreducible integrable highest weight modules are trivial.
\end{corollary}

\end{document}